\documentclass[11pt]{article}

\usepackage{amsmath,amssymb,amsthm}
\usepackage{graphicx}
\usepackage{subfigure}
\usepackage{enumerate}
\usepackage{fullpage}

\newenvironment{packed_enum}{
\begin{itemize}
  \setlength{\itemsep}{1pt}
  \setlength{\parskip}{0pt}
  \setlength{\parsep}{0pt}
}{\end{itemize}}

\newtheorem{theorem}{Theorem}

\newtheorem{problem}{Problem}
\newtheorem{proposition}{Proposition}
\newtheorem{question}{Question}
\newtheorem{conjecture}{Conjecture}

\newtheorem{lemma}[theorem]{Lemma}

\newcommand{\answerCommand}{}%
  {\renewcommand{\answerCommand}{#1\\}%
   \noindent\textbf{\answerCommand}%
   }{\\}%

  {\renewcommand{\answerCommand}{#1}%
   \noindent\textbf{\answerCommand}%
   }{\\}%

\title{On Polygons Excluding Point Sets}
\author{Radoslav Fulek\thanks{Ecole Polytechnique F\'ed\'erale de Lausanne. Email:~\texttt{radoslav.fulek@epfl.ch}}
 \and Bal\'{a}zs Keszegh\thanks{Alfr\'{e}d R\'{e}nyi Institute of Mathematics, Ecole Polytechnique F\'ed\'erale de Lausanne.
 Partially supported by grant OTKA NK 78439. Email:~\texttt{keszegh@renyi.hu}} \and
Filip Mori\'{c} \thanks{Ecole Polytechnique F\'ed\'erale de Lausanne. Email:~\texttt{filip.moric@epfl.ch}}
 \and Igor Uljarevi\'{c} \thanks{Matemati\v{c}ki fakultet Beograd. Email:~\texttt{mm07179@alas.matf.bg.ac.rs}}}
\date{}

\begin{document}

\maketitle

\pagenumbering{arabic} \setcounter{page}{1}

\begin{abstract}
By a polygonization of a finite point set $S$ in the plane we understand a simple  polygon 
 having $S$ as the set of its vertices. 
Let $B$ and $R$ be sets of blue and red points, respectively, in the plane such that $B\cup R$
 is in general position, and the convex hull of $B$ contains
$k$ interior blue points and $l$ interior red points.
 Hurtado et al. found sufficient conditions for the existence of a blue polygonization
that encloses all red points. We consider the 
dual question of the existence of a blue polygonization
 that excludes all red points  $R$. 
 We show that there is a minimal number $K=K(l)$, which is polynomial in $l$,
 such that one can always find a blue polygonization excluding
 all red points, whenever $k\geq K$. Some other related problems are also considered. 
\end{abstract}

\section{Introduction}

Let $S$ be a set of points in the plane in \emph{general position},
 i.e. such that no three points in $S$ are collinear. A \emph{polygonization} of $S$ is a simple (i.e. closed and non-self-intersecting)
polygon 
$P$ such that its vertex set is $S$.  

Polygonizations have received much attention recently.
One direction of research is to find good upper and lower bounds 
on the number of polygonizations of a given point set. 
This problem was raised in 1979 by Akl \cite{Akl}, who proved 
an exponential lower bound in the number of points $n$,
 and in 1980 by Newborn and Moser \cite{Newborn}, who conjectured an exponential upper bound.
The first exponential upper bound was established in 1982 by Ajtai, Chv\'atal, Newborn, and Szemer\'edi \cite{Ajtai}
for the number of crossing free graphs on $n$ points,
 and for polygonizations Sharir and Welzl \cite{Sharir} gave the upper bound $86.81^n$.

Another popular direction of research is to find a 
 polygonization of $S$, which maximizes or minimizes a given function. 
The exponential number of possible polygonizations of a given point set makes this problem in most cases hard.
Several such functions, some of which are geometric, e.g. the area of the polygon, were considered. 
Fekete in \cite{Fekete} considers the hardness of the problem of finding a polygonization of a point set that minimizes
or maximizes the enclosed area, and he proves that finding such polygonizations is NP-complete.
The other NP-complete problem is finding a polygonization of a point set that minimizes the perimeter \cite{NP}, which is the famous Euclidean Travelling Salesman problem.
This problem has a polynomial time approximation scheme \cite{Arora}.
Other functions are more of a combinatorial flavor,
 such as reflexivity (the number of reflex vertices of a polygonization), 
which measures how convex a polygon is.
 Algorithms finding a polygonization with not too many reflex vertices are presented in \cite{AFHMNSS03}
 and in \cite{reflex}. An important set of these problems originates in the study of red-blue separations.
% In \cite{Czyzowicz, Hurtado} they consider the following problem.
We have a set of blue points which has to be polygonalized 
and we have a set of red points, which restricts the polygonization in some way.
One natural kind of restriction was considered in \cite{Czyzowicz, Hurtado}.
We state it after introducing some definitions.

We say that a polygon $P$ \emph{encloses} a point set $V$ if all the points of $V$ 
belong to the interior of $P$. If all the points of $V$ belong to the exterior of $P$, then we say that $P$ \emph{excludes} $V$.
 Let $B$ and $R$ be disjoint point sets in the plane such that $B\cup R$ is in general position.
 The elements of $B$ and $R$ will be called \emph{blue} and \emph{red} points, respectively.
 Also, a polygon whose vertices are blue is a \emph{blue polygon}.
 A polygonization of $B$ is called a \emph{blue polygonization}.
 Throughout the paper in the figures we depict a blue point by a black disc, and a red point by 
 a black circle.
 
 Let $Conv(X)$ denote  the convex hull of a subset $X\subseteq \mathbb{R}^2$.
 By a \emph{vertex} of $Conv(X)$ we understand a 0-dimensional face on its boundary.
 We assume that all the red points belong to the interior of $Conv(B)$,
 since we can disregard red points lying outside $Conv(B)$ for the problems we consider.
 Let $n\geq 3$ denote the number of vertices of $Conv(B)$, 
 $k\geq 1$ the number of blue points in the interior of $Conv(B)$,
 and $l\geq 1$ the number of red points (which all lie in the interior of $Conv(B)$ by our assumption). 
 
In \cite{Czyzowicz, Hurtado} the problem of finding a blue polygonization that encloses the set 
$R$ was studied, and in \cite{Hurtado} Hurtado et al. showed that if the number of vertices of $Conv(B)$ is bigger then the number of red points,
 then there is a blue polygonization enclosing the set $R$. Moreover, they showed by a simple construction
 that this result cannot be improved in general.

% and the following tight result was obtained in \cite{Czyzowicz}:
% \begin{theorem}[\cite{Czyzowicz}]
% \label{thm:Hurtado}
% Let $B$ and $R$ be blue and red point sets in the plane such that $B\cup R$
%  is in general position and $R$ is contained in the interior of $Conv(B)$.
%  If the number of vertices of $Conv(B)$ is bigger then the number of red points,
%  then there is a blue polygonization enclosing the set $R$.
% \end{theorem}

Using the earlier terminology, we can say that they want to
 {\em maximize} the function defined to be the number of red points inside
 the polygon and prove that if we have enough blue points on the convex hull,
then this function can reach $l$. They also show that there always exists a polygonization 
of the blue points, which encloses at least half of the red points, i.e. if the number of red points is  
$l$, for a given point set, then the maximum of the function is always at least $l/2$.

We propose to study the problem of {\em minimizing} the same function, and as a starting point for that
we were interested in the following question.
% and conjecture the following.
%We are interested in a very similar restriction where the existence of a desired polygonization 
%is again not evident. Our main goal in this paper is to study the problem of finding a blue polygonization that 
%excludes all the points in $R$ (see Figure \ref{fig:IntroExample}).
%In particular, we solve the following problem:

%%%

%\begin{conjecture} \label{conj:main}
% { \rm (Total Polygonization)} 
%Let $B$ and $R$ be blue and red point sets in the plane such that $B\cup R$
% is in general position and $R$ is contained in the interior of $Conv(B)$.
% If the number of red points is $l$ and the number of blue points not on the convex hull is $k\ge 2l-1$
% then there is a blue polygonization excluding the set $R$.
%\end{conjecture}

%%%

%\begin{problem}{\rm (Total polygonization)}
%\label{prb:main}
%Is there for any $l>0$ a minimal number $K=K(l,n)$, depending only on $l$ and $n$,
%such that if  the number $k$ of blue points in the interior of $Conv(B)$ 
%satisfies $k\geq K$, then there exists a blue polygonization excluding all the red points? 
%\end{problem}

%It is already non-evident, if there exists any upper bound on the number of inner blue points that guarantees a polygonalization,
% which excludes all the red points (see Figure \ref{fig:IntroExample}), i.e. the problem is the following.

%In Section \ref{total} Theorem \ref{thm:main2} gives such an upper bound, which is polynomial only in $l$.

Is there for every $l$ a minimal number $K(l)$, such that having at least $K(l)$ 
blue points in the interior of $Conv(B)$ is a sufficient condition
for the existence of a blue polygonization of $|B|=n+k$ that excludes $R$, if $|R|=l$ (see Figure \ref{fig:IntroExample} for an illustration) ?
%The above question simply asks,
%whether having more interior blue points is always at least as good for finding a polygonization of $B$ excluding $R$ as having
%less of them.

In Section \ref{total} we answer the above question affirmatively by proving the following theorem.%in Section \ref{sec:total} 
\begin{theorem}
\label{thm:main2}
$ K(l)=O(l^4)$, i.e. there exists a $K_0(l)=O(l^4)$, so that 
if $|B|=n+k$, $k\geq K_0(l)$, $|R|=l$, and vertices of the convex hull of $B\cup R$ are $n$ blue points,
 then there exists a blue polygonization that excludes all the red points. 
\end{theorem}	

In Theorem \ref{thm:lowerbound} (Section \ref{total}) we give a complementary lower bound on $K(l)$, which shows $K(l)\ge 2l-1$.
Note that the property, which matters in the problem of enclosing red points \cite{Czyzowicz, Hurtado},
 is the number of blue vertices on the hull of $Conv(B)$,
while in our problem it seems that all that matters, is the number of blue points in the interior of $Conv(B)$.
Even though it is still possible that $K(l)$ depends on $n$, the evidence we have suggests that this is not the case.
Moreover, we conjecture that our lower bound on $K(l)$ is tight.

\begin{conjecture} \label{conj:main}
{ \rm (Total Polygonization)} 
Let $B$ and $R$ be blue and red point sets in the plane such that $B\cup R$
is in general position and $R$ is contained in the interior of $Conv(B)$.
If the number of red points is $l$ and the number of blue points not on the convex hull is $k\ge 2l-1$
then there is a blue polygonization excluding the set $R$.
\end{conjecture}

In the previous problem we were interested in finding a polygonization for \emph{all} blue points. 
As a variant to this problem, it is natural to ask, what is the smallest number of inner blue points,
that we can use in a polygonization visiting all the vertices of $Conv(B)$, so that all red points are excluded. 
One can easily see that it is not always possible to find a simple blue $(2n-1)$-gon excluding all the red points,
no matter how large the number of inner blue points is, as soon as $|R|=l\geq n$ holds ($n$ is the number of vertices of $Conv(B)$).
Indeed, if we put a red point very close to every side of the convex hull $Conv(B)$, we make sure that in any blue 
 polygonization we cannot use any side of the convex hull,
 which implies that at least $n$ inner blue points must be sometimes used. 
We remark that the same construction was used in \cite{Hurtado}
to justify the tightness of their main result mentioned above.

So, it is natural to ask the following question. 

%\begin{problem}{\rm (Minimal polygonization)}
%\label{prb:main2}
%Let $n$ denote the number of vertices of $Conv(B)$,
%$k$ the number of blue points in the interior of $ConvB$ and $l$ the number of red points.
% We assume that all red points belong to the interior of $Conv(B)$. If $n$ and $l$ are fixed, is it true that for every large enough $k$ 
%there exists a simple blue $2n$-gon avoiding all red points and such that all the $n$ vertices
%of $Conv B$ are also its vertices? 
Is there for any $l>0$ a minimal number $K'=K'(l,n)$
such that if the number $k$ of blue points in the interior of $Conv(B)$ 
satisfies $k\geq K'$, then there exists a blue polygonization 
of a subset of $B$ of size at most $2n$ excluding all the red points? 
%\end{problem}

Section \ref{sect:minimal} deals with this question, and we answer it affirmatively in Theorem \ref{thm:igor}.

Finally, in Section \ref{sect:redblue} we tackle the following closely related problem.  
Given $n$ red and $n$ blue points in general position, we want to draw a polygon separating the two sets,
 with minimal number of sides. Theorem \ref{thm:redblue} shows that we can always find such a separating polygon
 using $3\lceil n/2\rceil$ sides and we also show that it is not always possible to draw such a polygon with less than $n$ sides.

\begin{figure}[h]
\centering
  \subfigure[]{\label{fig:IntroExample} 
		\includegraphics[scale=0.55]{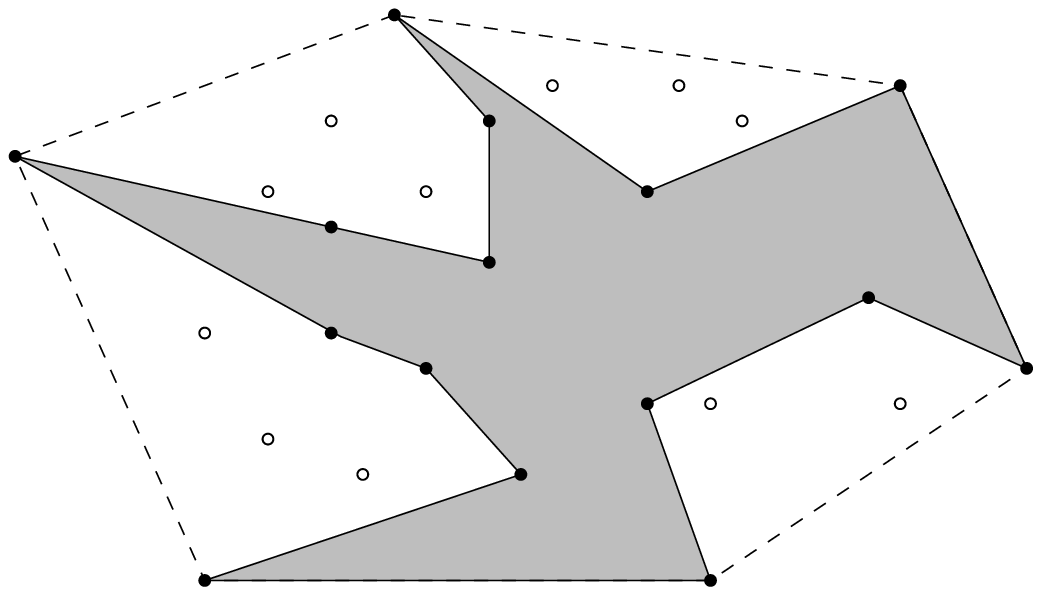}
        \hspace{5mm} 
      }
  \subfigure[]{\label{fig:partition}     
		\includegraphics[scale=0.55]{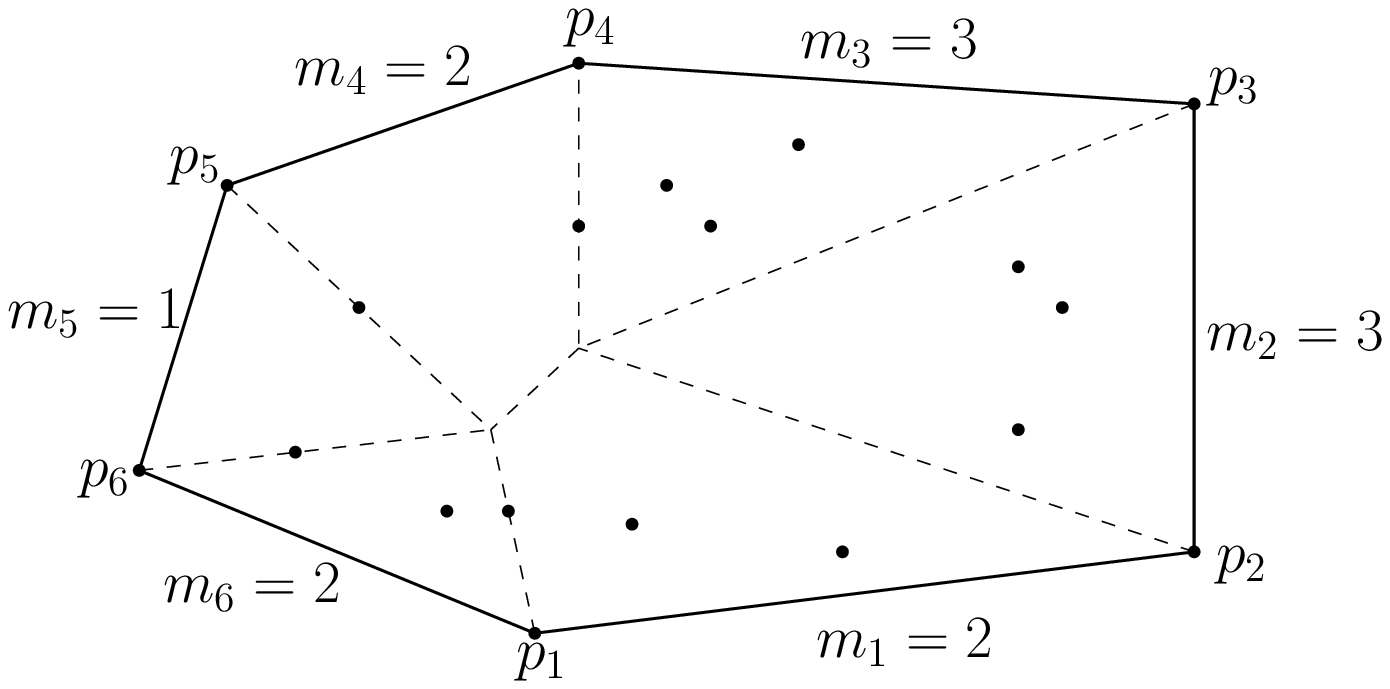}
	
	}
	
	\caption{(a) A blue polygonization excluding all the red points, (b) A partition guaranteed by Lemma \ref{lemma:partition}}
\end{figure}

\section{Preliminary results}

In this section we present several lemmas that we will use throughout the paper. Let us recall that $B$ and $R$
 denote sets of blue and red points in the plane. We will assume that they are in general position, i.e. the set $B\cup R$ does not contain three collinear points.
We will need the following useful lemma by Garc\'{i}a and Tejel \cite{Garcia} (see Figure \ref{fig:partition}).

\begin{lemma}\label{lemma:partition}{\rm (Partition lemma)}
Let $P$ be a set of points in general position in the plane and assume that $p_1,p_2,\dots,p_n$ are the vertices of the $Conv(P)$ 
and that there are $m$ interior points. Let $m=m_1+\dots+m_n$, where the
$m_i$ are nonnegative integers. Then the convex hull of $P$ can be partitioned into $n$ convex
polygons $Q_1,\dots,Q_n$ such that $Q_i$ contains exactly $m_i$ interior points (w.r.t. $Conv(P)$) and $p_ip_{i+1}$ is an edge of $Q_i$.
(Some interior points can occur on sides of the polygons $Q_1,\dots,Q_n$ and for those points we decide which region
they are assigned to.)
\end{lemma}

The next lemma is a straightforward application of Lemma \ref{lemma:partition}.
It will be used as the main ingredient in the solution of the Minimal polygonization problem,
but we also find it interesting on its own.

\begin{lemma}{\rm (Alternating polygon lemma)}
\label{lemma:alternating}
If $|B|=|R|=n$ and the blue points are vertices of a convex $n$-gon, while all the red points are in the
interior of that $n$-gon, then there exists a simple alternating $2n$-gon, i.e. a $2n$-gon in which any two consecutive
vertices have different colors.
\end{lemma}

\begin{proof}
Let $B=\{b_1,\dots,b_n\}$ and $R=\{r_1,\dots,r_n\}$. If we apply Lemma
 \ref{lemma:partition} with $m_1=\dots=m_n=1$, we will get a partition of the blue $n$-gon into $n$ 
convex parts $Q_1,\dots,Q_n$, each of them containing exactly one red
 point and such that $b_ib_{i+1}$ is an edge of $Q_i$ for each $i$. 
Without loss of generality we can assume that $r_i\in Q_i$.
 Then one can easily see that $b_1r_1b_2r_2\dots b_nr_n$ is the desired alternating polygon
 (see Figure \ref{fig:alternating}).
\end{proof}

\begin{figure}[h]
\centering
  \subfigure[]{\label{fig:alternating}     
		\includegraphics[scale=0.55]{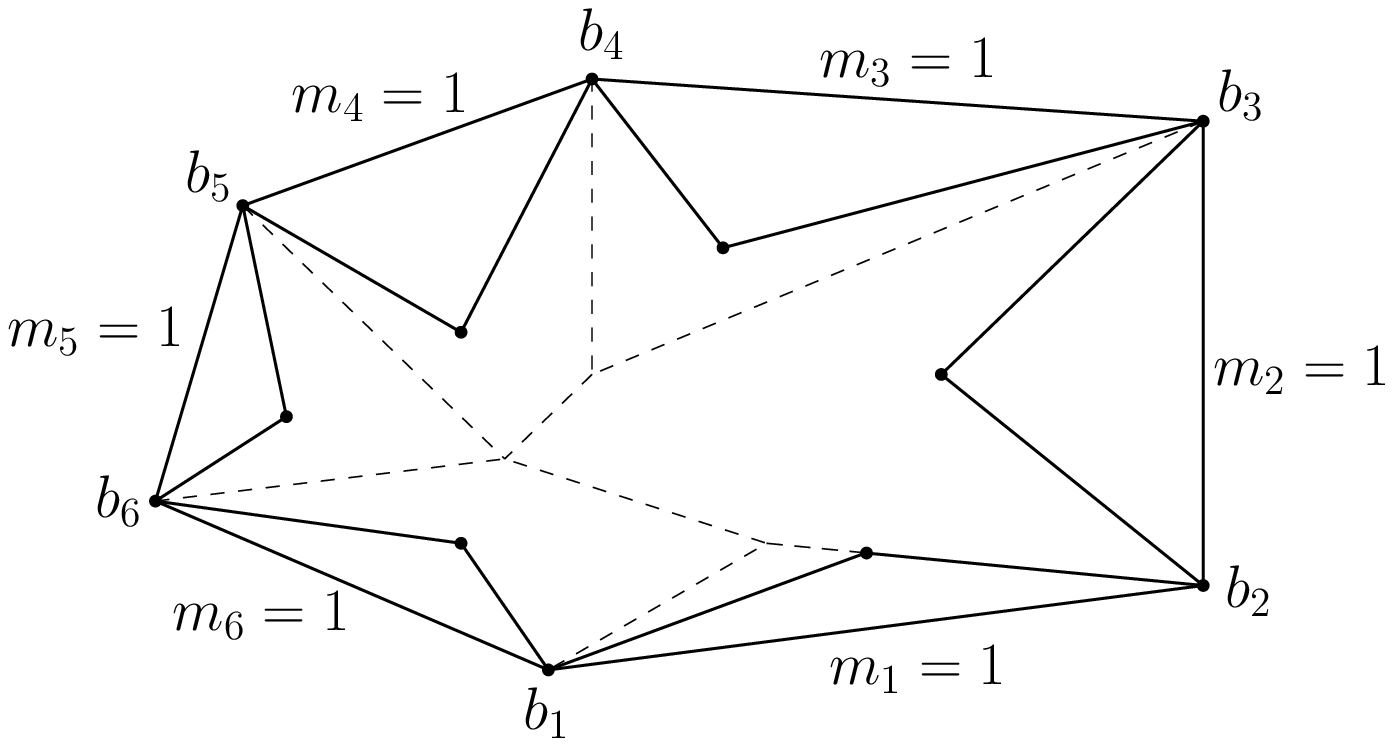}
       
        \hspace{5mm} 
      }
  \subfigure[]{\includegraphics[scale=0.55]{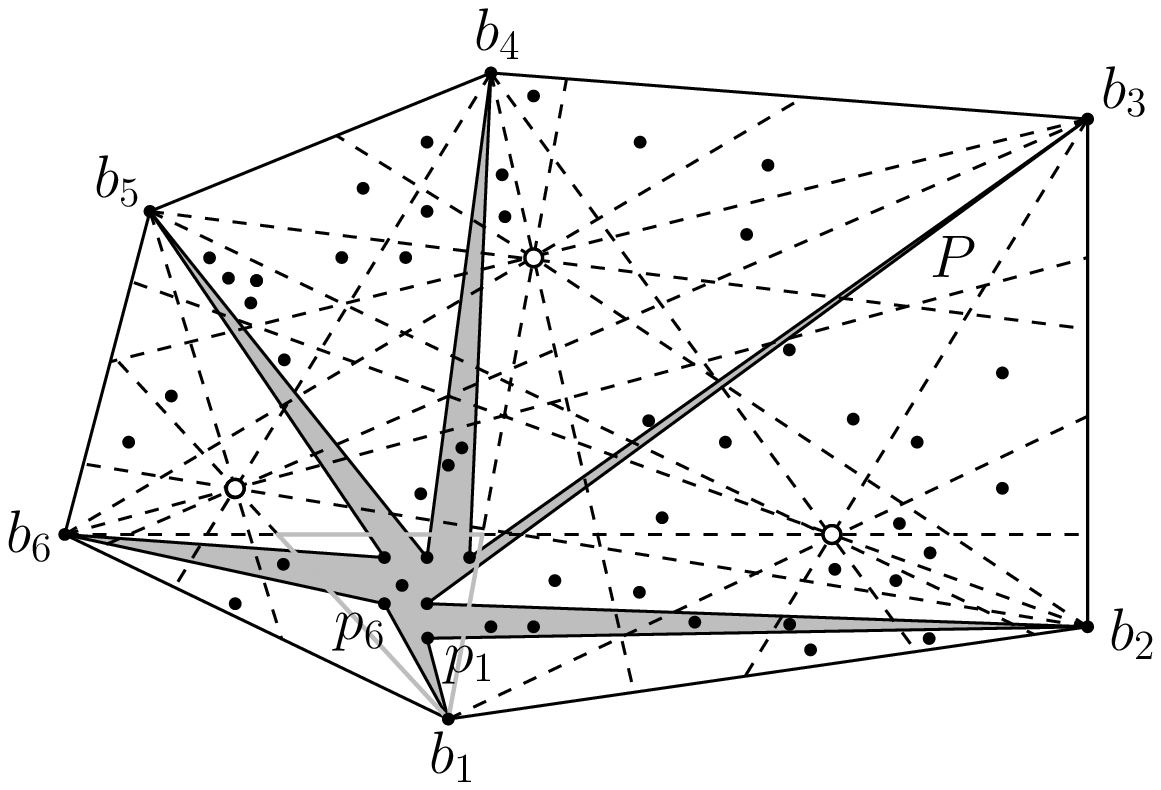}
	\label{fig:igor}
	}
	
	 \caption{(a) A polygon obtained by Lemma \ref{lemma:alternating}, (b) Alternating polygon} 
\end{figure}

In the solution of the Total polygonization problem we will be making a 
polygon by concatenating several polygonal paths obtained by the following proposition,
 which is rather easy. 

\begin{proposition}
\label{lemma:path}
Let $S$ be a set of $n$ points in the plane in general position and $p$ and $q$ two points from the convex hull of $S$.
Then one can find a simple polygonal path whose endpoints are $p$ and $q$ and whose vertices are the $n$ given points.
\end{proposition}

\begin{proof}
Since the points $p$ and $q$ belong to the convex hull,
 there exists a point $r$ in the plane such that angle $\angle prq$
 is smaller than $\pi$ and contains all
the points from $S$. We can also suppose that $r$ does not belong to
 a line determined by two points from $S$. Let us denote the points from $S$ by
 $p_1,p_2, \ldots, p_n$, so that $p_1=p$, $p_n=q$ and $\angle prp_2<\angle prp_3<\ldots<\angle prp_{n-1}<\angle prq$
 (see Figure \ref{fig:path}).
 Then the polygonal path $p_1p_2\dots p_n$ is simple as required. 
\end{proof}

\begin{figure}[h]
    \label{fig:path}
	\includegraphics[scale=0.55]{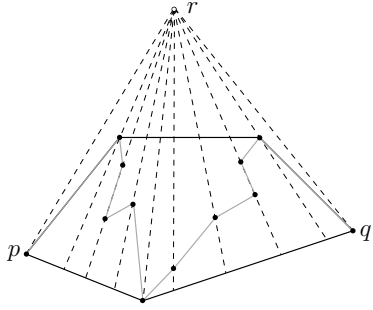}
	\centering
	\caption{A polygonal path from $p$ to $q$ passing through all vertices of $P$}
\end{figure}

In order to obtain by our method a bound on $K(l)$ ($|R|=l$) we need to take care of the situation,
when the convex hull $Conv(B)$ contains too many vertices.
For that sake we have the following proposition, which can be established quite easily.

\begin{proposition}
\label{prop:smallerHull}
There exists a subset $B'$ of $B$ of size at most $2l+1$, containing only the vertices of $conv(B)$,
so that all the red points are contained in $conv(B')$.
\end{proposition}

\begin{proof}
Let $B'$ denote a minimal subset of vertices of $B$ so that the conclusion of the claim is true.
If $|B'|\geq 2l+2$, by the pigeonhole principle we can find three consecutive vertices $p,q,r$ on $conv(B')$,
so that the triangle $pqr$ does not contain a red point. Thus, $B'\setminus \{q\}$ contradicts the choice of $B'$.
\end{proof}

\section{Minimal polygonization} \label{sect:minimal}

Now we are ready to give the first result.
As it was mentioned in Introduction,
the number $2n$ in the next theorem cannot be improved.

\begin{theorem}
\label{thm:igor}
If $|B|=n+k$, $|R|=l$, $k\geq K'(l,n)=O(n^3l^2)$ and the convex hull of $B$ contains $k$ blue vertices in its interior, 
then there exists a simple blue polygonization of a subset of $B$ of size at most $2n$, which contains all the vertices
of the convex hull of $B$, and  excludes 
all the red points.
%and such that all the $n$ vertices
%of the convex hull of $B\cup R$ are also its vertices. 
\end{theorem}

\begin{proof}

Let $b_1,\ldots, b_n$ be the vertices of the convex hull.
Consider all the lines determined by one blue point from the convex hull and one red point. It is easy to see that by drawing these $nl$
 lines the interior of $Conv(B)$ is divided into no more than $(nl)^2$  2-dimensional regions.
Since we have at least $K'(l,n)$ interior blue points, it follows that there is a region that contains at least $n$ blue points
(see Figure \ref{fig:igor}).

Let $p_1, \ldots, p_n$ be blue points that lie inside one region.
By Lemma \ref{lemma:alternating} it follows that there exists a  simple $2n$-polygon $P$ whose vertices
are taken alternatingly from the sets $\{b_1,\ldots, b_n\}$ and $\{p_1,\ldots, p_n\}$.
 It is easy to see by the proof of Lemma \ref{lemma:alternating} that this $2n$-gon satisfies 
the following property: For each point $x$ from the interior of the $2n$-gon there is a blue point $b_i$ such that the segment $b_ix$  is
   entirely contained in the $2n$-gon. 

Without loss of generality we can assume that 
$P=b_1p_1\ldots b_np_n$. We claim that $P$ does not contain any red point in its interior. Suppose the contrary, i.e. there exists 
a red point $r$ in the interior of $P$.
Then there exists a blue vertex $b_i$ such that the segment $b_ir$ lies in the interior of $P$. Hence, the line $l$ through $b_i$ and $r$
 intersects the line segment $p_{i-1}p_{i}$ (where $p_0=p_n$),
 which cannot be true because all the points $p_1,\ldots, p_n$
 lie in the same closed half-plane defined by $l$.
 This contradiction finishes the proof.    
\end{proof}

\section{Total polygonization}\label{total}

\label{sec:total}

The aim of this section is to prove the main result, which is stated in Theorem \ref{thm:main2},
 about sufficient conditions for the existence of a blue polygonization that excludes all the red points.

%First, we will introduce a lemma whose proof is a bit technical, but based on a simple idea. We use Erd\H{o}s-Szekeres theorem and the pigeonhole principle 
%to find a subset of blue points with a nice structure. This structure makes it possible to form a partition into convex parts. Then we find a polygonal path in each convex
%part and finally we merge all these paths.  

By a {\it wedge} with $z$ as its apex point we mean a convex hull of
 two non-collinear rays emanating from $z$.
We define an ($l$-){\em zoo} $\mathcal{Z}=(B,R,x,y,z)$ as a set $B=B(\mathcal{Z})$ of blue and $R=R(\mathcal{Z})$, $|R|=l$, red points with two
 special blue points $x=x(\mathcal{Z})\in B, y=y(\mathcal{Z})\in B$ and a special point $z=z(\mathcal{Z})$ (not necessarily in $B$ or $R$) such that:

\begin{packed_enum}
\item every red point is inside $Conv(B)$
\item $x,y$ are on the boundary of $Conv(B)$
\item every red point is contained in the wedge 
$W=W(\mathcal{Z})$ with apex $z$ and boundary rays $zx$ and $zy$
\end{packed_enum}

We denote by $B^*=B^*(\mathcal{Z})$ the blue points inside $W'=W'(\mathcal{Z})$,
the wedge opposite to $W(\mathcal{Z})$. We refer to the points in $B^*$ as to special blue points.
We imagine $x$ and $y$ being on the $x$-axis (with $x$ having smaller $x$-coordinate than $y$)
and $z$ being above it (see Figure \ref{fig:pie}),
and we are assuming that when we talk about objects being below each other in a zoo.

A {\em nice partition} of an $l$-zoo is a partition of $Conv(B)$
into closed convex parts $P_0,P_1,\ldots, P_m$,
 for which there exist pairwise distinct special blue 
points $b_1,\ldots, b_{m}\in B^*$ (we call $b_0=x$ and $b_{m+1}=y$)
 such that for every $P_i$ we have that (see Figure \ref{fig:nicepart}):
\begin{packed_enum}
\item no red point is inside $P_i$ i.e. red points are on the boundaries of the parts
\item $P_i$ has $b_i$ and $b_{i+1}$ on its boundary
\end{packed_enum}

\begin{figure}[h]
    \centering
    \subfigure[]{\label{fig:pie}
        \includegraphics[scale=0.3]{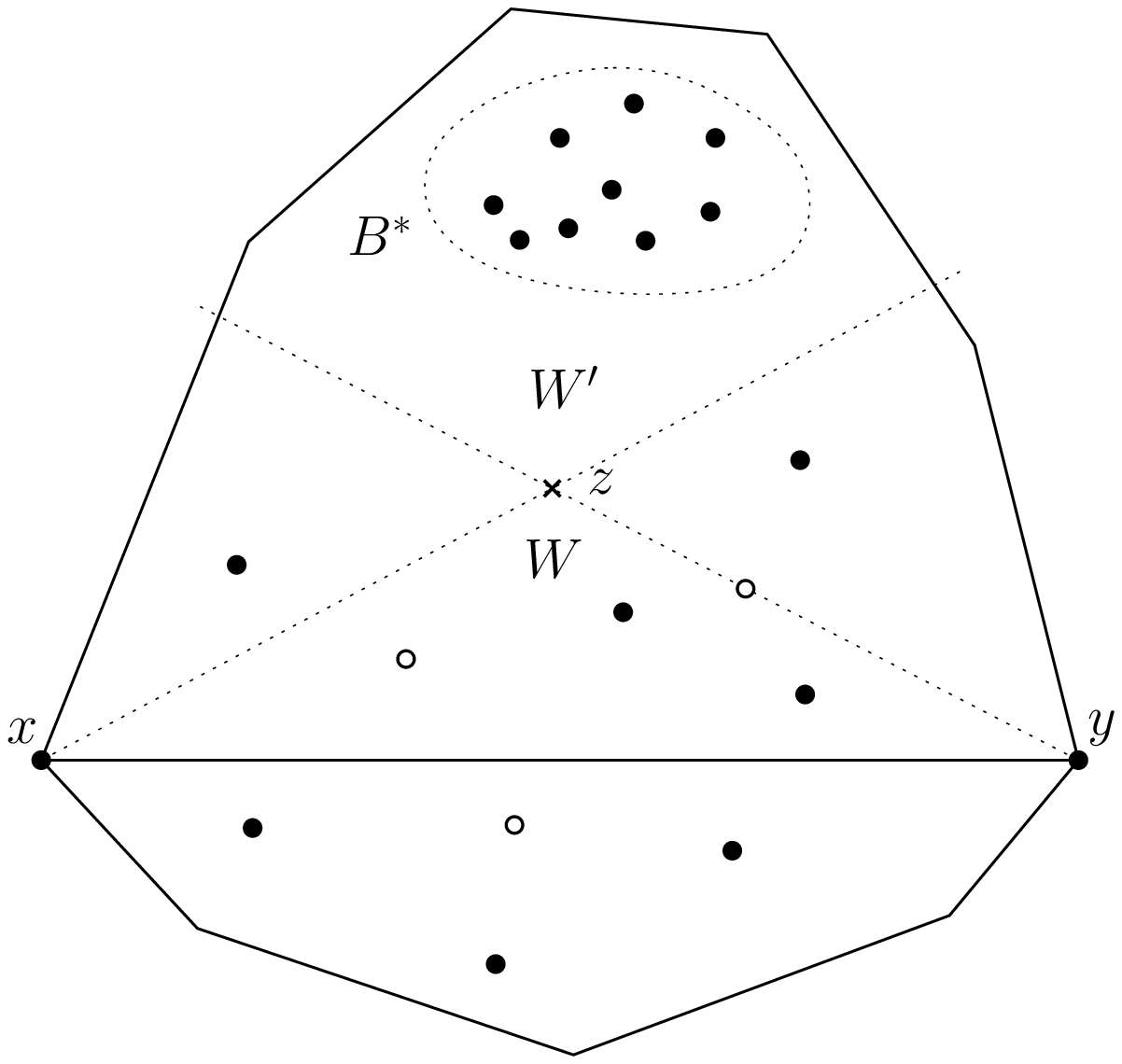}
        \hspace{5mm}}
    \subfigure[]{\label{fig:nicepart}
				\includegraphics[scale=0.3]{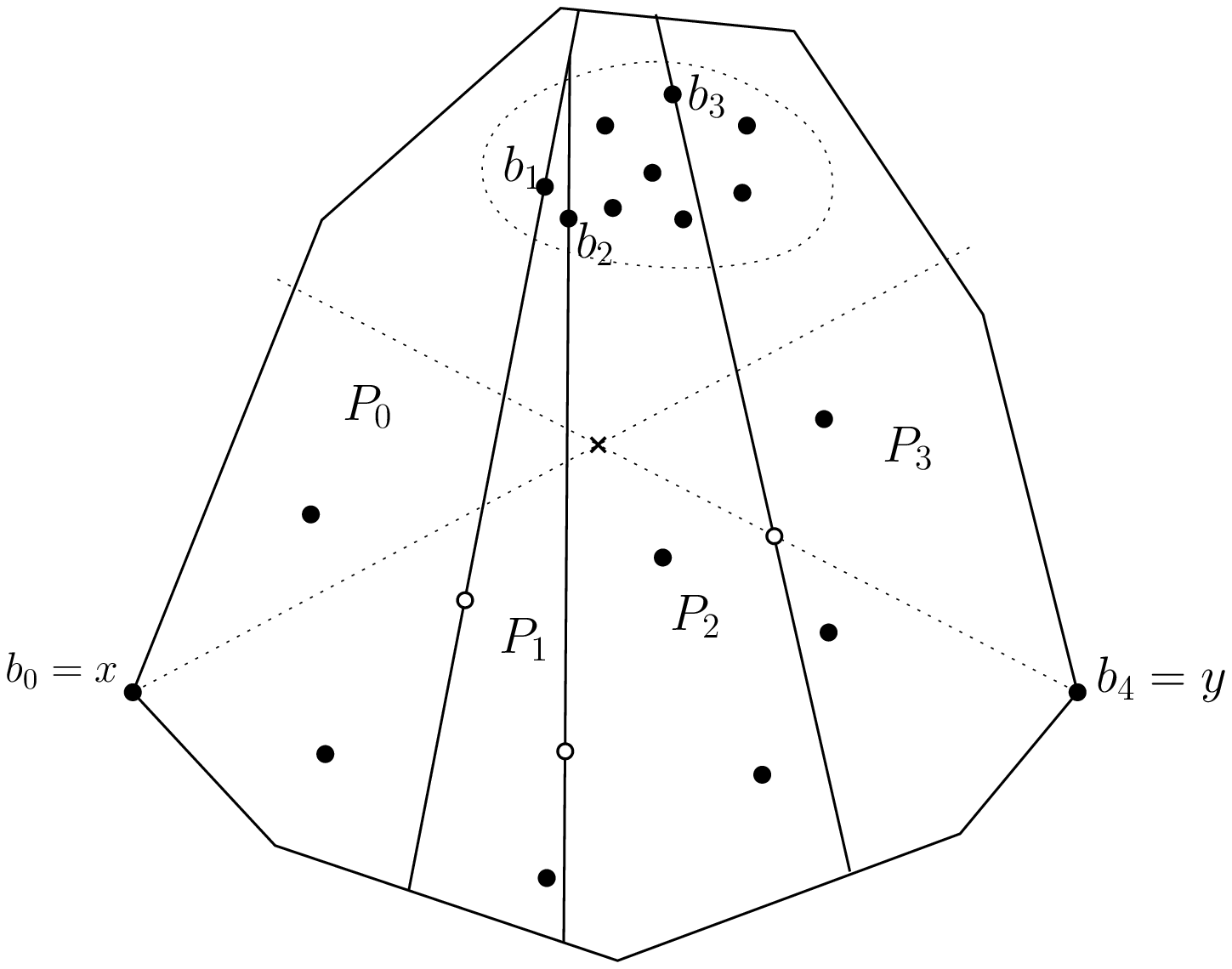}
}
   \caption{(a) $3$-zoo, (b) Nice partition of $3$-zoo into $4$ parts}
   \label{}
\end{figure}

\begin{proposition}
\label{prop:zooprop}
Given a zoo $\mathcal{Z}$ with a nice partition, we can draw a polygonal path
 using all points of $B=B(\mathcal{Z})$ with endpoints $x(\mathcal{Z})$ and $y(\mathcal{Z})$
 s.t. all the red points are below the polygonal path.
\end{proposition}
\begin{proof}
Indeed, by applying Proposition \ref{lemma:path} in each $P_i$, $0\leq i\leq m$,
 with $P_i\cap B$ as $S$, $b_{i}$ as $p$ and $b_{i+1}$ as
 $q$ we get polygonal paths, which can be concatenated in order to form the desired polygonal path.

\end{proof}

\begin{lemma}
\label{lemma:capCup}
Given an $l$-zoo $\mathcal{Z}$, if $B^*=B^*(\mathcal{Z})$ contains a 
blue $y$-monotone convex chain of size $2l-1$, then it has a nice partition.
\end{lemma}

\begin{proof}
 Let $C=\{c_1, c_2, \ldots, c_{2l-1}\}$ denote a $y$-monotone blue convex chain of size $2l-1$,
 so that $y(c_1)<y(c_2)<\ldots <y(c_{2l-1})$. 
If $l>1$, without loss of generality, by the $y$-monotonicity  we can assume that the interior
 of $Conv(\{c_i, c_{i+1},\ldots, c_j\} )$ 
is on the same side of the line $c_{i}c_{j}$, for all $1\leq i<j\leq 2l-1$, as an unbounded portion of a positive part
of the $x$-axis. 

The special points of the nice partition will be always points of this chain.
We start by taking $Q_{-1}=Conv(B)$.
Then, we recursively define the partition $P_0,P_1,\ldots,P_i,Q_i$
 and points $b_1,b_2,\ldots, b_{i+1}\in B^* $ such that for each $P_i$ the two properties
 needed for a nice partition hold and the remainder $Q_i$ 
of the zoo is a convex part with $b_{i+1}$ and $y$ on its boundary.
We define $R_i=R\cap int(Q_i)$, $C_i=C\cap int(Q_i)$ 
and either $R_i$ is empty or $|C_i|\ge 2|R_i|-1$ and then $t_i$ denotes the common tangent of $Conv(C_i)$
 and $Conv(R_i)$, which has the point $y$ 
and the interior of $Conv(C_i)$ and $Conv(R_i)$ on the same side
(see Figure \ref{fig:lemma7def} for an illustration). We maintain the following:\\

\begin{tabular}{rl}
($\star$) & If $R_i$ is nonempty, then $t_i$ intersects the boundary of $Q_i$ in a point \\
& with higher $y$-coordinate than $b_i$.\\ 

\end{tabular}
\\

In the beginning when $i=-1$, $|C_i|\ge 2|R_i|-1$ and ($\star$) holds trivially.

In a general step, $P_0,P_1,\ldots,P_i,Q_i$ being already defined we do the following.

If $Q_i$ does not contain red points inside it,
taking $P_{i+1}=Q_i$ and $m=i+1$ finishes the partitioning. The convex set 
$P_m=Q_i$ has $b_{m+1}=y$ and $b_{m}=b_{i+1}$ on its boundary. Hence, the two necessary properties hold for $P_m$. 

Otherwise, let $P_{i+1}$ be the intersection of $Q_i$ with the closed half-plane defined by $t_i$, which contains
$x$. Trivially, there is no red point inside it. As $t_i$ intersects the boundary of $Q_i$ in a point with higher $y$-coordinate then $b_{i+1}$, we have that $P_{i+1}$ has $b_{i+1}$ on its boundary.
 Let $b_{i+2}$ denote the blue point lying on $t_i$, trivially $b_{i+2}$
 is on the boundary of
 $P_{i+1}$ too. It is easy to see that the point $b_{i+1}$ has either the lowest or the highest $y$-coordinate among the points 
in $C_{i}$. We define $Q'_i$ as the closure of $Q_i\setminus P_{i+1}$, 
$R'_i=R\cap int (Q'_i)$, $C'_i=C\cap int (Q'_i)$,
 and $t'_i$ denotes 
the common tangent of $Conv(C'_i)$ and $Conv(R'_i)$,
 which has the point $y$ and the interior of $Conv(C'_i)$ and $Conv(R'_i)$
 on the same side.
If $t'_i$ cannot be defined then $R'_i$ is empty and the next step will be the final step, we just take $Q_{i+1}=Q'_i$.
\begin{figure}[h]
    \centering
    \subfigure[]{\label{fig:lemma7def}
        \includegraphics[scale=0.43]{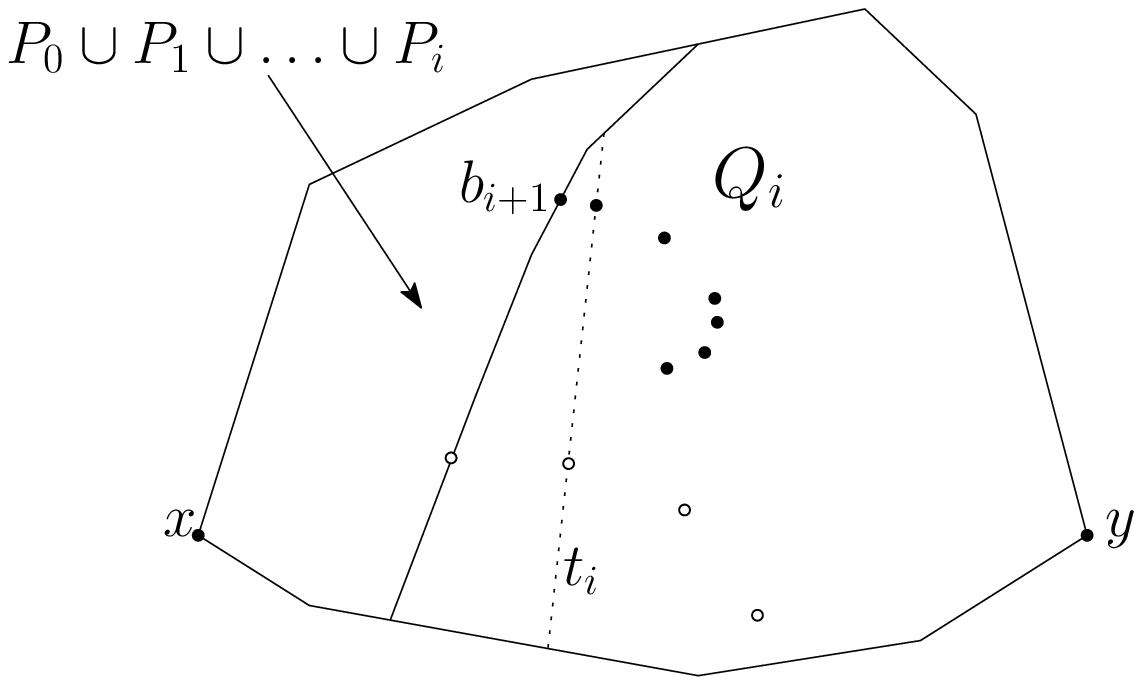}
        \hspace{0mm}}
    \subfigure[]{\label{fig:lemma7a}
        \includegraphics[scale=0.43]{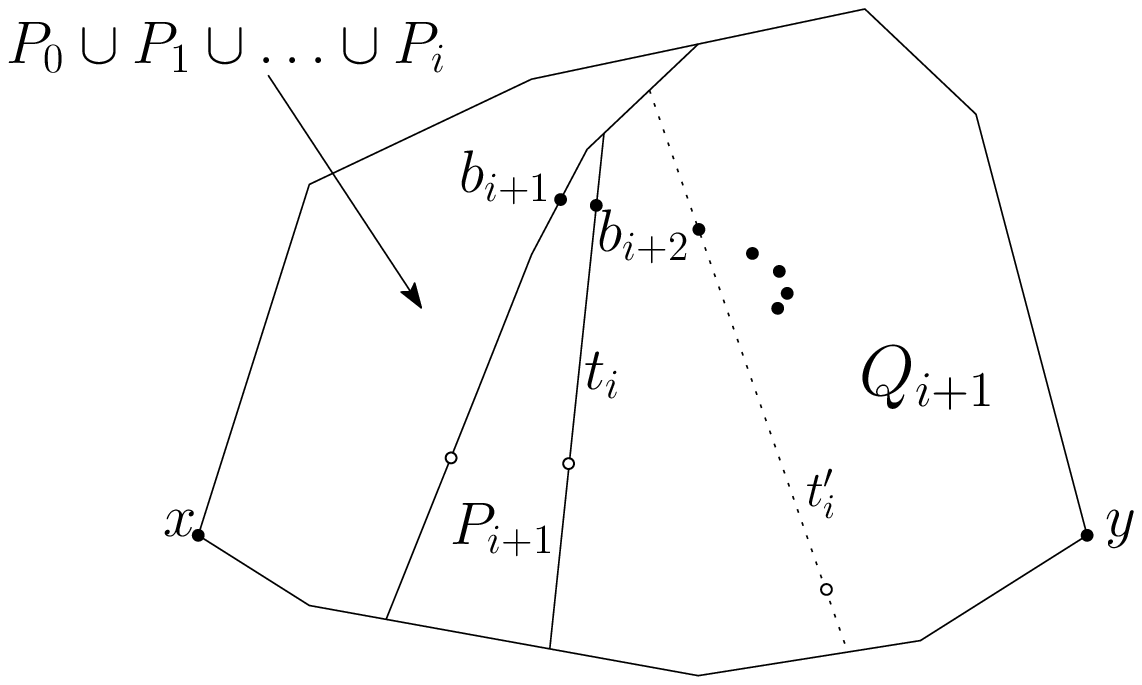}
        \hspace{0mm}}
    \subfigure[]{\label{fig:lemma7b}
				\includegraphics[scale=0.43]{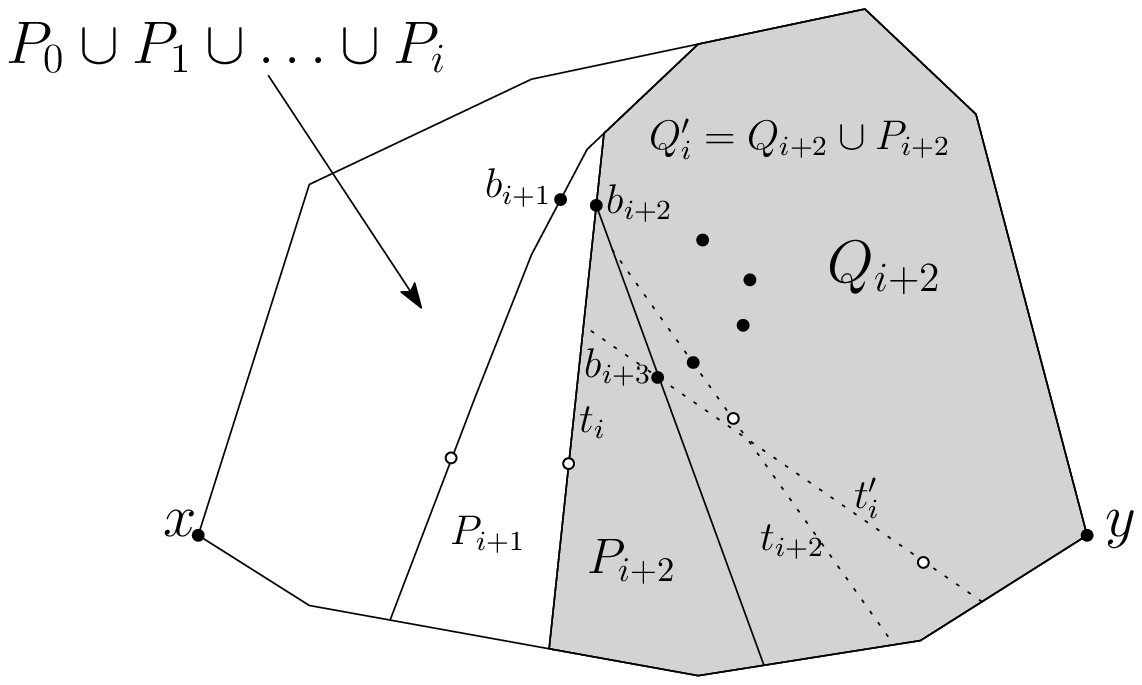}
}
   \caption{(a) a general step of the recursion continuing with (b) case (i) or (c) case (ii)}
   \label{}
\end{figure}

\begin{enumerate}[(i)]
\item \label{item:lemma7(i)}
If $t'_i$ intersects the boundary of $Q'_i$ in a point with 
higher $y$-coordinate than $b_{i+2}$ then ($\star$) will 
hold in the next step so we can finish this step by taking $Q_{i+1}=Q'_i$ (see Figure \ref{fig:lemma7a}).
\item \label{item:lemma7(ii)}
If $t'_i$ does not intersect the boundary of $Q'_i$ in a point with higher $y$-coordinate than $b_{i+2}$ then 
we do the following (see Figure \ref{fig:lemma7b}).
Denote by $b_{i+3}$ the blue point on $t'_i$. Now $P_{i+2}$ is defined as the intersection of $Q'_i$ and the half-plane 
defined by the line $b_{i+2}b_{i+3}$ and containing $x$. It is easy to see that $P_{i+2}$ 
does not contain red points in its interior, and it has 
both $b_{i+2}$ and $b_{i+3}$ on its boundary. We finish this step by taking  $Q_{i+2}$ as the closure of
 $Q'_i\setminus P_{i+2}$.
It remains to prove that in the next step property ($\star$) holds.

First, observe that $b_{i+3}$ has either the lowest or the highest $y$-coordinate
 among the points in $C_{i+2}$. Moreover, it is easy to see that it has to be the lowest
 one otherwise we would end up in Case (\ref{item:lemma7(i)}). Thus, the blue point on the new tangent $t_{i+2}$ is
 a point of the chain $C$ that is higher then $b_{i+3}$. Then the intersection of $t_{i+2}$ with the boundary of $Q_{i+2}$
 must be a point with higher $y$-coordinate than $b_{i+3}$ as needed.
\end{enumerate}

The condition $|C_i|\ge 2|R_i|-1$ holds by induction. Indeed, 
in each step the number of remaining red points
decreases by $1$, while the number of remaining blue points decreases at most by $2$
 except the last step when we never have Case (\ref{item:lemma7(ii)}), and thus,  the number of remaining blue points  decreases also just by $1$.

% By the $y$-monotonicity of the chain $C$, in  case (\ref{item:lemma7(ii)}), the region $P_{i+2}$
% does contain any point from $C$ in its interior.
% Thus, as we have $|C|\geq 2l-1$, in the above definition, it can never happen that $C_i$ is empty and $R_i$ is not.
% Indeed, we cannot remove more than two blue points from $C$ without the removal of a point from $R$.
% Hence, the only problematic case is $|C|= 2l-1$. However, it means that we can end up in case 
%(\ref{item:lemma7(ii)}) at most $l$ times.
\end{proof}

The next lemma is a variant of the previous one,
and it is the key ingredient in the proof of the main theorem in this section.

\begin{figure}[h]
\centering
\includegraphics[scale=0.55]{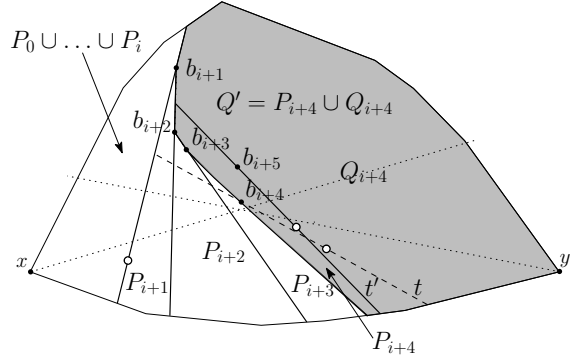}
\caption{A general step of the recursion in Lemma \ref{lemma:capCup2}, $s=4$} 
\label{fig:lemma8}
\end{figure}

\begin{lemma}
\label{lemma:capCup2}
Given an $l$-zoo $\mathcal{Z}$ if $B^*=B^*(\mathcal{Z})$ contains at least $\Omega  (l^2)$ blue points, then it has a nice partition.
\end{lemma}

\begin{proof}
We can suppose that in $B^*$ there is no $y$-monotone convex chain of size $2l-1$, because otherwise we can apply Lemma \ref{lemma:capCup} in order
to get a desired nice partition.

We start by taking $Q_{-1}=Conv(B)$ and $C=C_{-1}=B^*$.
As in Lemma \ref{lemma:capCup} we recursively define the partition $P_0,P_1,\ldots,P_i,Q_i$ and points $b_1,b_2,\ldots,b_{i+1}$ such that for each $P_i$ the two properties needed for a nice partition hold and the remainder
 $Q_i$ of the zoo $\mathcal{Z}$ is a convex part with $b_{i+1}$ and $y$ on its boundary. We define $R_i=R\cap int(Q_i)$, $C_i=C\cap int(Q_i)$.

In a general step, $P_0,P_1,\ldots,P_i,Q_i$ being already defined we do the following.

If $Q_i$ does not contain red points inside it,
taking $P_{i+1}=Q_i$ and $m=i+1$ finishes the partitioning. The convex set 
$P_m=Q_i$ has $b_{m+1}=y$ and $b_{m}=b_{i+1}$ on its boundary. Hence, the two necessary properties of a nice partition hold for $P_m$.

Otherwise, we again define $t$, the common tangent of $Conv(C_i)$ and $Conv(R_i)$ which has the point $y$ and the interior of $Conv(C_i)$ and $Conv(R_i)$ on the same side of $t$. If $t$ intersects the boundary of $Q_i$ in a point with higher $y$-coordinate than $b_{i+1}$ then we can finish this step as in Lemma \ref{lemma:capCup} by taking $b_{i+2}$ as the blue point on $t$, $P_{i+1}$ as the intersection of $Q_i$ with the closed half-plane defined by $t$ and  $Q_{i+1}$ as the closure of $Q_i\setminus P_{i+1}$.

If $t$ does not intersect the boundary of $Q_i$ in a point with higher $y$-coordinate 
than $b_{i+1}$, then we define $b_{i+1},b_{i+2},\ldots,b_{i+s}$, $b_{i+s}\in t$, to be the consecutive vertices of $Conv(C_i)$,
 for which the segments with one endpoint $x$ and the other being any of these points, do not cross $Conv(C_i)$.
 As this is a $y$-monotone convex chain with $s$ vertices, we have that $s<2l-1$. 

We obtain the regions $P_{i+1}, P_{i+2}, \ldots, P_{i+s-1}$ (see Figure \ref{fig:lemma8}),
 by cutting $Q_{i}$ successively with the lines through the pairs
 $b_{i+1}b_{i+2}, b_{i+2}b_{i+3}, \ldots, b_{i+s-1}b_{i+s}$ (in this order).
 Evidently, these regions satisfy the property needed for a nice partition.
Let $Q'$ stand for the remaining part of $Q_{i}$ (the gray region in Figure \ref{fig:lemma8}).
 Furthermore, $R'=R\cap int(Q')$ and $C'=C\cap int(Q')$. We define $t'$ to be the common tangent of $Conv(C')$ and $Conv(R')$ 
which has the point $y$ and the interior of $Conv(C')$ and $Conv(R')$ on the same side.
We define $b_{i+s+1}$ to be the blue point on $t'$ and $P_{i+s}$ 
to be the intersection of $Q'$ with the closed half-plane defined by $t'$ and containing $x$.
Again $P_{i+s}$ satisfies the property needed for a nice partition, as it has
 $b_{i+s+1}$ and $b_{i+s}$ on its boundary.
Indeed, otherwise $t'$ would not intersect the boundary of $Q'$ in a point with higher $y$-coordinate than $b_{i+s}$,
in which case $t'$ could not be the tangent to $Conv(C')$ and $Conv(R')$, a contradiction.

Observe that $C_{i+s}$ contains all points of $C_i$ except $b_{i+2},b_{i+3},\ldots, b_{i+s+1}$.
Because of that, if we proceed in this way recursively, in each step the number of remaining red points
decreases by $1$, while the number of remaining blue points decreases by $s<2l-1$.
Thus, if originally, we had $(2l-2)l+1$ blue points in $B^*$, we can proceed until the end thereby finding a nice partition of  
$\mathcal{Z}$.
\end{proof}

Having the previous lemma, we are in the position to prove  Theorem \ref{thm:main2}.

% \begin{theorem}
% \label{thm:main2}
% $ K(l)=O(l^4)$, i.e. there exists a $K_0(l)=O(l^4)$, so that 
% if $|B|=n+k$, $k\geq K_0(l)$, $|R|=l$, and vertices of the convex hull of $B\cup R$ are $n$ blue points,
%  then there exists a blue polygonization that excludes all the red points. 
% \end{theorem}	

\begin{proof}[Proof of Theorem \ref{thm:main2}.] %({\bf Theorem \ref{thm:main2}})
First, by Proposition \ref{prop:smallerHull} we obtain 
a subset $B'$, $|B'|=m$, %($Conv(B')$ is bounded by thick segments in the left part of Figure \ref{fig:MainThm}) 
of the vertices of $Conv(B)$ of size at most $2l+1$,
so that $R\subseteq Conv(B')$. 
%Let $b_0',\dots,b_{m-1}'$ be the vertices of $Conv(B')$ listed in a cyclic order.
Let $b_0', b_1', \ldots, b_{m-1}'$ denote the blue points in $B'$ listed according to their cyclic order on the boundary of $Conv(B')$.

% $0\leq i\leq m-1$ (indices are taken modulo $m$),
% and by all the pairs consisting of one blue vertex from $B'$ and one red point.
%It is easy to see that by drawing these $2m$ 
%lines $Conv(B)$ is divided into no more than $(ml+{m \choose 2})^2=O(\min \{l^4,n^2l^2\})$ regions.
%By the hypothesis of the lemma we assume that there are $\Omega (\min \{l^6,n^2l^4\})$ interior blue points in $Conv(B)$.
%It follows that there is a region $S$ that contains at least $\Omega (l^2)$ blue points. Let $Q$ stand for the blue points in $S$.

%Suppose that the convex hull of $B' \cup Q$ contains at least
%$l$ vertices, which belong to $Q$. These $l$ vertices have to appear in a cyclic order on the boundary
%between two consecutive vertices of $Conv(B')$. Indeed, if that is not the case,
%we can divide $Q$ into two non-empty parts by a line through a point in $B'$ and $R$ (contradiction).
%Thus, we can apply Lemma \ref{lemma:specialCase} in order 
%to obtain a polygonization of $B$ that excludes $R$, i.e. all the red points.

\begin{figure}[h]
\centering
\includegraphics[scale=0.45]{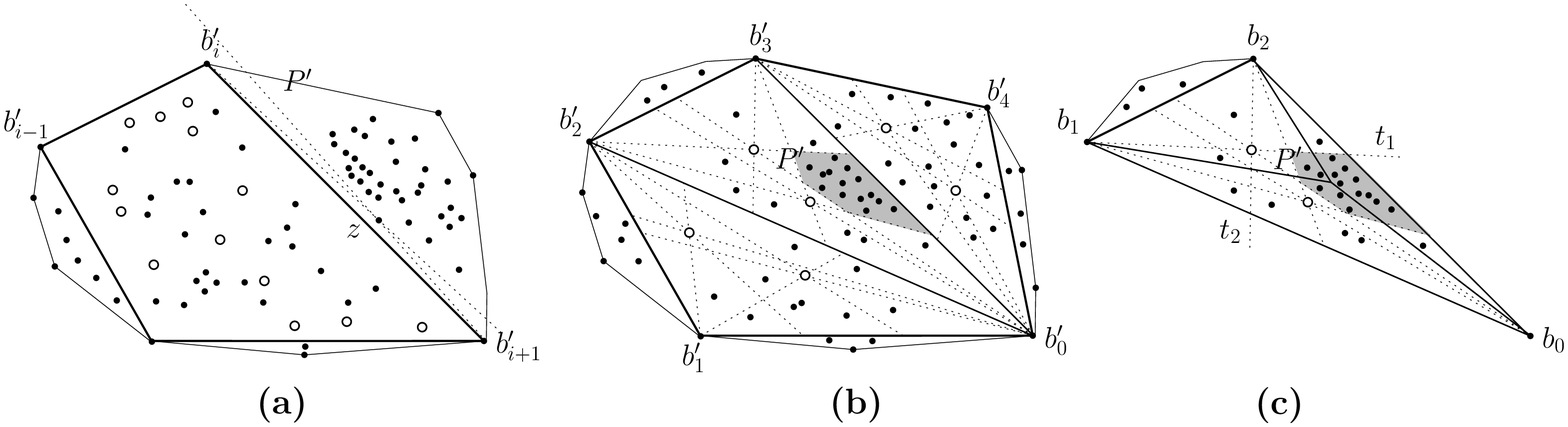}
\caption{Partition of $Conv(B)$} 
\label{fig:MainThm}
\end{figure}

%Let $c_0, c_1 \ldots c_t, c_{t+1} \ldots c_{u-1}$, $u= O(l)$, denote the vertices of $Conv(B' \cup Q)$ listed in a cyclic order,
%so that $c_1, \ldots c_t\in Q$. For the convenience we allow $t$ to be zero, which means that $Q\subseteq Conv(B')$.
%Clearly, either $S\subseteq Conv(B')$, or $S\not\subseteq Conv(B')$.

First, we suppose that $Conv(B')$ does not contain $\Omega(l^4)$ points in its interior (see Figure \ref{fig:MainThm} (a)).
It follows, that there is a convex region $P'$ containing $\Omega(l^3)$ blue points,
which is an intersection of $Conv(B)$  with a closed half-plane $T$ defined by a line through two
consecutive vertices $b_i'$ and $b_{i+1}'$, for some $0 \leq i < m$ (indices are taken modulo $m$),
 on the boundary of $Conv(B')$,
such that $T$ does not contain the interior of $Conv(B')$. 
%In fact, at this point requiring  $\Omega(l^2)$ blue points in $P'$ would suffice, but in the other case a weaker assumption in the hypothesis of the theorem on the number of blue points would not be enough.
Let $B''$ denote the set of vertices of $Conv(P')$ except $b_i'$ and $b'_{i+1}$.
Observe that we have an $l$-zoo $\mathcal{Z}$ having $B(\mathcal{Z})=B\setminus B''$,
$R(\mathcal{Z})=R$, $b_i'$ and $b_{i+1}'$ as $x(\mathcal{Z})$ and $y(\mathcal{Z})$, respectively.
By the general position of $B$ we can take $z(\mathcal{Z})$ to be a point very close to the line segment $b_i'b_{i+1}'$,
so that $B^*(\mathcal{Z})$ contains $\Omega(l^2)$ blue points.
Thus, by Lemma \ref{lemma:capCup2} we obtain a nice partition of $Z$. Hence, by Proposition \ref{prop:zooprop} 
we obtain a blue polygonal path $Q$ having $B\setminus B''$ as a set of vertices.
The desired polygonal path is obtained by concatenating the path $Q$ with the convex chain
formed by the points in $B''\cup\{b_i',b_{i+1}'\}$.

Thus, we can suppose that $Conv(B')$ contains $\Omega(l^4)$ points in its interior (see Figure \ref{fig:MainThm} (b)).
Let $R_i$ denote the intersection of $R$ with the triangle $b_0'b_i'b_{i+1}'$, for all $1\leq i <m-1$.
For each triangle $b_0'b_i'b_{i+1}'$ we consider the  lines through all the pairs $r$ and $b$,
such that $b=b_0',b_i'$ or $b_{i+1}'$ and $r\in R_i$.
For each $i$, $1\leq i <m-1$, these lines partition the triangle $b_0'b_i'b_{i+1}'$ into  $O(|R_i|^2)$ 2-dimensional regions.
Hence, by doing such a partition in all the triangles  $b_0'b_i'b_{i+1}'$
we partition $Conv(B')$ into $O(\sum_{i=1}^{m-2}|R_i|^2)=O(|R|^2)$ regions, each of them fully contained
in one of the triangles $b_0'b_i'b_{i+1}'$. It follows that one of these regions, let us denote it by $P'$, contains at least $\Omega(l^2)$ blue
points. Clearly, $P'$ is contained in a triangle $b_0'b_i'b_{i+1}'$, for some  $1\leq i <m-1$.

For the convenience we rename the points $b_0',b_i',b_{i+1}'$ by  $b_0,b_1,b_2$ in clockwise order.
We apply Partition Lemma (Lemma \ref{lemma:partition}) on the triangle $b_0b_1b_2$, so that we obtain 
a partition of the triangle $b_0b_1b_2$ into three convex polygonal regions $P_0',P_1', P_{2}'$ (in fact triangles),
such that each part contains $\Omega(l^2)$ blue points belonging to $P'\cap P_j'$, for all $0\leq j \leq 2$, and has $b_jb_{j+1}$ as a boundary segment.
We denote by $P_0,P_1,P_{2}$  the parts in the partition of $Conv(B)$,
which is naturally obtained as the extension of the partition of $b_0b_1b_2$,
so that $P_j$, $P_j\supseteq  P_j'$, has $b_jb_{j+1}$ (indices are taken modulo $3$) either as a boundary edge or as a diagonal.

In what follows we show  that in each $P_j$, $0\leq j\leq 2$, %, which contains at most $l$ red points, and a set of $\Omega (l^2)$ blue points separated
 %by a wedge from all red points in $Q_i$, such 
we have an $l_j$-zoo $\mathcal{Z}_j$, $l_j\leq l$, with $b_j$ as $x(\mathcal{Z}_j)$ and $b_{j+1}$ and $y(\mathcal{Z}_j)$, respectively,
and with $\Omega(l^2)$ blue points in $B^*(\mathcal{Z}_j)$.

First, we suppose that there exists a red point in $P_j'$.
 We take
$z(\mathcal{Z}_j)$ to be the intersection of two tangents $t_1$ and $t_2$ from $b_j$ and $b_{j+1}$, respectively,
 to $Conv(R \cap P_j')$ that have $Conv(R \cap P_j')$ and $b_jb_{j+1}$ on the same side.
 Clearly, $P'$ has to be contained in one of four wedges defined by $t_1$ and $t_2$ (see Figure \ref{fig:MainThm} (c)).
  However, if $P'$ is not contained in the wedge defined by $t_1$ and $t_2$,
 which has the empty intersection with the line through $b_j$ and 
$b_{j+1}$, either $P_{j+1}$ or $P_{j-1}$ cannot have a non-empty intersection with $P'$ (contradiction).
Thus, $B^*(\mathcal{Z}_j)$ of $\mathcal{Z}_j$ contains at least $\Omega(l^2)$ blue points.

Hence, we can assume that $P_j'$ does not contain any red point.
In this case, by putting $z$ very close to $b_jb_{j+1}$, so that $z\in b_0b_1b_2$, we can make sure, that
the corresponding wedge above the line $b_jb_{j+1}$ contains all the blue points in $P'$. 

 Thus, in every $P_j$, $0\leq j\leq 2$, we have $\mathcal{Z}_j$ with $b_j$ and $b_{j+1}$ as $x(\mathcal{Z}_j)$ and $y(\mathcal{Z}_j)$,
 respectively,
the set of blue points in $P_j$ as $B(\mathcal{Z}_j)$, and the set of red points in $P_j$ as $R(\mathcal{Z}_j)$.
By using Proposition \ref{prop:zooprop} on a nice partition of $\mathcal{Z}_j$ obtained by Lemma \ref{lemma:capCup2}
we obtain a polygonal path using all the blue points in $P_j$
which joins $b_j$ and $b_{j+1}$, and which has all the red points in $P_j$ on the "good" side.
Finally, the required polygonization is obtained by concatenating  
the paths obtained by Lemma \ref{lemma:capCup2}.
\end{proof}

The polynomial upper bound on $K(l)$ in the previous theorem is complemented with the following lower bound construction, which still leaves
a huge gap. 
%In fact, all the evidence we have suggests that $K(l)$ is linear.

\begin{theorem} \label{thm:lowerbound}
For arbitrary $n\ge 3,l\ge 1$ and $k\le 2l-2$ there is a set of points $B\cup R$ (as before $|B|=n+k$,
 $|R|=l$ and the set of  vertices of the convex hull of $B\cup R$ consists of  $n$ blue points),
 for which there is no polygonization of the blue points, which excludes all the red points. 
\end{theorem}

\begin{proof}
For fixed $n$ and $l\ge 1$ and $k=2l-2$ we define the set $B$ as follows (see Figure \ref{fig:lowerbound} for an illustration).
 We put two blue points $x$ and $y$ on the $x$-axis, $x$ being left from $y$.
 In the upper half-plane we put $n-2$ blue points $Z=\{z_1,z_2,...,z_{n-2}\}$ close to each other
such that  $Z'=\{x,y\} \cup Z$ are in convex position. Let us call a vertex in $Z$ a $z$-vertex.
Furthermore, we put $l-1$ blue points (not necessarily in convex position) to the interior of $Conv(Z')$ close to the $z$-vertices, we call them $b$-vertices.
 Next, we put $l$ red points in the interior of $Conv(Z')$,
all below the lines $xz_{n-2}$ and $yz_1$ such that together with
$x$ and $y$ they form a convex chain $xr_1r_2\ldots r_ly$. 
Finally, for each segment $r_ir_{i+1}$, we put a blue point $l_i$ a bit below its midpoint.
We call these $l$-vertices (lower blue vertices). This way we added $l-1$ more blue points.
Suppose that there exists a polygon $P$ through all the blue points excluding all the red points.
 Starting with a $b$-vertex we take the vertices of the polygon one by one until we reach an $l$-vertex,
 say $l_i$.
 The vertex preceding $l_i$ on the polygon cannot be $x$, as in this case $r_1$ would be in the interior of 
$P$, and similarly it cannot be $y$ as then $r_l$ would be in the interior of $P$.
 If it is a $z$-vertex then  $r_i$ or $r_{i+1}$ is inside $P$. Thus, it can be only a $b$-vertex.
Now, the vertex following $l_i$ on the polygon cannot be neither $x,y$ nor an $l$-vertex as in all of these cases  $r_i$ or $r_{i+1}$ would be inside $P$.
For the same reason it cannot be a $z$-vertex. Hence, it must be a $b$-vertex. Now, we find the next $l$-vertex on the polygon. Again, the vertex before and
after it must be a $b$-vertex.
Proceeding this way we see that every $l$-vertex is preceded and followed by a $b$-vertex.
As we have other vertices on the polygon too, it means that the number of $b$-vertices is at least one more than the number of $l$-vertices, a contradiction.
\end{proof}

The construction for $l=5, k=8$ is in Figure \ref{fig:lowerbound}.
We remark that the same construction without the $z$-vertices shows that Lemma \ref{lemma:capCup2} is not true,
if we require that $W(\mathcal{Z})$ contains at least $l-1$ blue points. The proof is similar.
Regarding the exact values for $K(l)$ for a small $l$, it is not hard to check, that 
for $l=1,2$: $K(1)=1$ and $K(2)=3$.

\begin{figure}[h]
\centering
\subfigure[]{
	\includegraphics[scale=0.4]{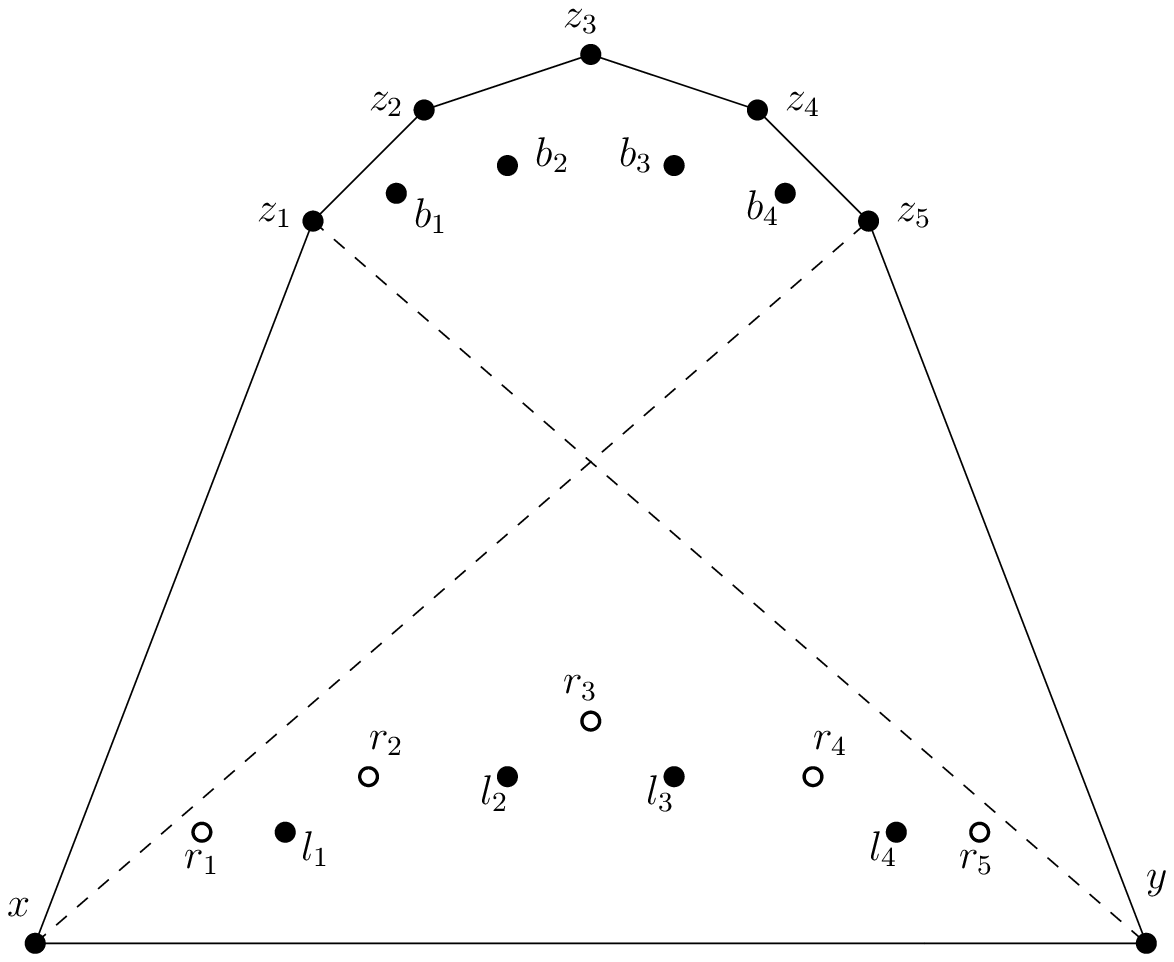}
	\label{fig:lowerbound}
	\hspace{5mm}
	}
\subfigure[]{
	\includegraphics[scale=0.4]{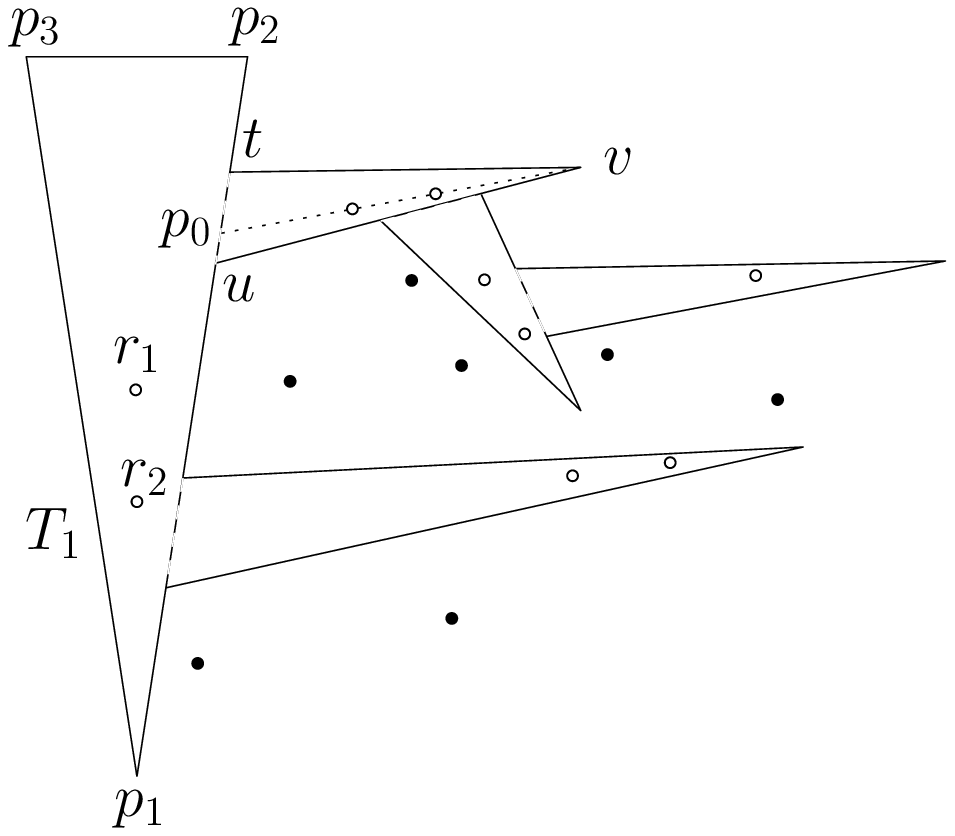}
	\label{fig:rbs}
	\hspace{5mm}
 }
 \subfigure[]{
	\includegraphics[scale=0.4]{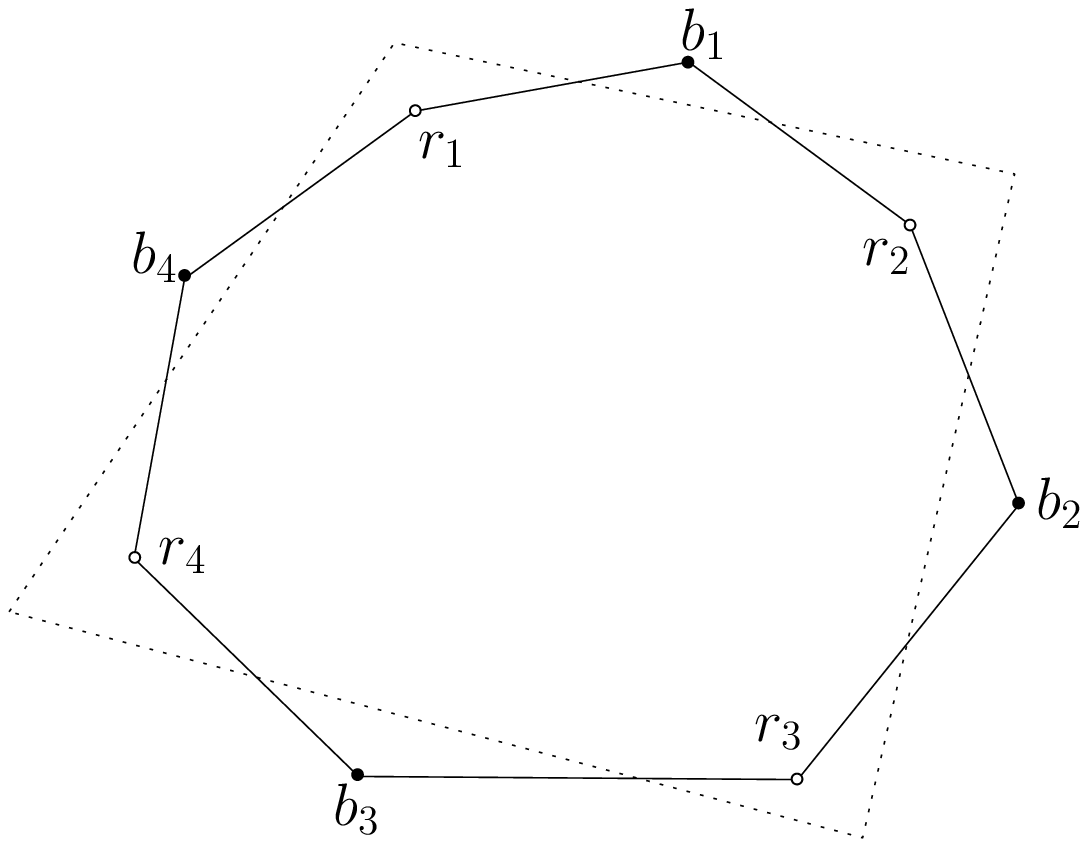}
	\label{fig:rbe} 
	}
\caption{
(a) The lower bound construction, (b) Construction of the red-blue separation, and (c) The lower bound construction for red-blue separation}
\end{figure}

\section{Red-blue separation} \label{sect:redblue}

\begin{theorem} \label{thm:redblue}
Let $B$ and $R$ be sets of $n$ blue and $n$ red points in the plane in general position. Then there exists a simple polygon with at most $3\lceil n/2\rceil$ sides that 
separates blue and red points. 

Also, for every $n$ there are sets $B$ and $R$ that cannot be separated by a polygon with less than $n$ sides. 
\end{theorem}

\begin{proof}
Let $R=\{r_1,\dots,r_n\}$, where $x(r_1)\leq x(r_2)\leq\dots\leq x(r_n)$. By choosing the coordinate system appropriately we can assume that $x(r_1)=x(r_2)=0$. 
Due to the general position we can find numbers $a,b>0$ large enough so that for certain $c>0$ the triangle $T_1$ (see Figure \ref{fig:rbs}) with vertices $p_1=(0,-a)$, $p_2=(c,b)$, $p_3=(-c,b)$ 
has the following properties:\begin{itemize}
\item $T_1$ contains $r_1$ and $r_2$ and does not contain any other red or blue points
\item all the lines $r_{2i-1}r_{2i}$ ($i=2,3,\dots$) intersect the boundary of $T_1$
\end{itemize}

We will proceed by enlarging the polygon $T_1$ adding to it in each step three new vertices so that the new polygon contains the next pair of red points and no blue points.
Since the line $r_3r_4$ intersects the boundary of $T_1$ at some point $p_0$ we can find two points
 $t$ and $u$ on the boundary of $T_1$ close enough to $p_0$ and a point $v$ on the line $r_3r_4$ close 
to one of the points $r_3, r_4$, so that the triangle $tuv$ can be joined with $T_1$ thereby creating a new polygon $T_2$ that contains the points $r_1,r_2,r_3,r_4$,
 and does not contain any other red or blue point.
Notice that the condition requiring, that any line determined by two consecutive red points intersects the boundary of $T_2$, is still satisfied,
 since it was already true for $T_1$. Observe that $T_2$ has 6 vertices.

 We can continue in this way by adding the pairs $r_i,r_{i+1}$ for $i=5,7,\dots,2\lfloor n/2\rfloor-1$ one by one. In the end we  get a polygon
 $T_{\lfloor n/2\rfloor}$, that contains
all the red points, except $r_n$ in case of odd $n$,  has $3\lfloor n/2\rfloor$ vertices, and does not contain any blue point. 
If $n$ is even, we are done. Otherwise we can add in the same manner three new 
vertices to $T_{\lfloor n/2\rfloor}$ in order to include $r_n$ as well.

Finally, let us show that we cannot always find a separating polygon with
 less than $n$ sides. Let $r_1,b_1,r_2,b_2,\dots,r_n,b_n$ be the vertices of a convex
 $2n$-gon appearing in that order on the circumference and set $R=\{r_1,\dots,r_n\}$ and
 $B=\{b_1,\dots,b_n\}$ (see Figure \ref{fig:rbe}). Let $P$ be any polygon that separates the two sets. Obviously, 
each of  the $2n$ segments $r_1b_1,b_1r_2,\dots,r_nb_n,b_nr_1$ must be intersected by a side of $P$.
 Since one side of $P$ can intersect simultaneously at most two of these segments,
 it follows that $P$ must have at least $n$ sides. 
\end{proof}

Note that the previous claim holds even if we have $n$ red points and arbitrary number of blue points.

\section{Concluding remarks}

Theorem \ref{thm:main2} in Section \ref{sec:total} proves the existence of a total blue polygonization excluding
 red points if we have enough inner blue points. We showed an upper bound on $K(l)$, the needed number of inner blue points, that is polynomial,
 but likely not tight.
%, which is also the case for our polynomial bound in both $n$ and $l$.
 We conjecture that the upper bound is $2l-1$, which meets the lower
 bound in Theorem \ref{thm:lowerbound}. 
%If $l\le 2$ then a case-analysis shows that the conjecture holds.
 If finding
 the right values of $K(l)$ and $K(n,l)$ for all $l,n$
 turns out to be out of reach, it is natural to ask the following.

\begin{question}
What is the right order of magnitude of $K(l)$ and $K'(n,l)$ ? 
%for which Theorems \ref{thm:main2} and \ref{thm:igor} hold?
\end{question}

One could obtain a better upper bound on $K(l)$ e.g. by proving Lemma \ref{lemma:capCup2}
with a weaker requirement on the number of blue points in $W(\mathcal{Z})$, which we suspect is possible.	
 
\begin{question}
Does Lemma \ref{lemma:capCup2} still hold, if we require only to have $\Omega(l)$ points in $W(\mathcal{Z})$, instead of $\Omega (l^2)$?
\end{question}

As mentioned in the introduction, the problem can be phrased as minimizing the number of enclosed red points. In this paper we only gave a bound on the number of inner blue points, 
beyond which this function is $0$. It would be natural to ask in general, what is the minimum of the function for a given point set.
 If Conjecture \ref{conj:main} holds, then it implies that having $l$ red points and $k<2l-1$ blue points, we can always find a
 polygonization with at least $\left\lfloor \frac{k+1}{2}\right\rfloor$ red points excluded. Indeed, just select, arbitrarily, this many red 
points and find a polygonization which excludes them. Unfortunately, at the moment we 
cannot find a lower bound meeting this conjectured upper bound.

\begin{question}
Can we find a point set $B\cup R$ for any $n,l$ and $k<2l-1$ such that every polygonization of $B$ excludes at most $\left\lfloor \frac{k+1}{2}\right\rfloor$ red points?
\end{question}

Finally, the bounds we have on the minimal number of sides for the red-blue separating polygon do not meet.

\begin{problem}
Improve the bounds $n$ or/and $3\lceil n/2\rceil$ in Theorem \ref{thm:redblue}.
\end{problem}

%The last problem we mention is less closely related to this topic but we think that it is an interesting question. 
%One of the main ingredients for the proofs in this paper was Lemma \ref{lemma:partition} proved in 
%\cite{Garcia}. We conjecture its following generalization to an arbitrary dimension.
%
%\begin{conjecture}
%Let $P$ be a set of points in general position in $\Bbb R^d$ and assume that $f_1,f_2,\dots,f_n$ are the facets of the $Conv(P)$ 
%and that there are $m$ interior points. Let $m=m_1+\dots+m_n$, where the
%$m_i$ are positive integers. Then the convex hull of $P$ can be partitioned into $n$ convex
%polytopes $Q_1,\dots,Q_n$ such that $Q_i$ contains exactly $m_i$ interior points and $f_i$ is a facet of $Q_i$.
%\end{conjecture}

\section{Acknowledgements}
We would like thank J\'anos Pach for helpful discussions regarding the subject and providing us with the relevant references. Also, we acknowledge Ferran Hurtado's communication with us about the state of the art of the problem.

\end{document}